\documentclass[12pt]{amsart}
\usepackage{amsfonts}
\usepackage{amssymb}
\usepackage{amsthm}
\usepackage{dsfont}
\usepackage{color}
\usepackage{url}
\usepackage{enumerate}
\usepackage[normalem]{ulem}

\newcommand{\red}[1]{{\color{red}#1}}

\newtheorem{theorem}{Theorem}[section]
\newtheorem{lemma}[theorem]{Lemma}
\newtheorem{proposition}[theorem]{Proposition}
\newtheorem{corollary}[theorem]{Corollary}

\newtheorem{fact}[theorem]{Fact}
\newtheorem{example}[theorem]{Example}
\newtheorem{problem}[theorem]{Problem}

\newtheorem{remark}[theorem]{Remark}

\newcommand{\on}{\operatorname}

\newcommand{\N}{\mathbb N}
\newcommand{\Q}{\mathbb Q}
\newcommand{\R}{\mathbb R}
\newcommand{\Z}{\mathbb Z}
\newcommand{\Co}{\mathfrak c}
\newcommand{\ve}{\varepsilon}

\newcommand{\res}{\upharpoonright}
\newcommand{\la}{\langle}
\newcommand{\ra}{\rangle}

\newcommand{\mc}{\mathcal}
\newcommand{\Perf}{\on{Perf}}
\newcommand{\osc}{\on{osc}}
\newcommand{\cl}{\on{cl}}
\newcommand{\HCRP}{\mathrm{HCRP}}
\newcommand{\HSOP}{\mathrm{HSOP}}
\newcommand{\diam}{\on{diam}}
\newcommand{\card}{\on{card}}
\renewcommand{\int}{\on{int}}

\author{Marek Balcerzak}
\address{Institute of Mathematics,
         Lodz University of Technology, al. Politechniki 8,
         93-590 \L\'od\'z,
         Poland}
         \email{marek.balcerzak@p.lodz.pl}
         
\author{Tomasz Natkaniec}
\address{Institute of Mathematics, Faculty of Mathematics, Physics and Informatics,
University of Gda\'{n}sk, 80-308 Gda\'{n}sk, Poland}
\email{tomasz.natkaniec@ug.edu.pl}

\author{Piotr Szuca}
\address{Institute of Mathematics, Faculty of Mathematics, Physics and Informatics,
University of Gda\'{n}sk, 80-308 Gda\'{n}sk, Poland}
\email{piotr.szuca@ug.edu.pl}

 \title{Point-set games and functions with the hereditary small oscillation property}

\begin{document}
 \begin{abstract}
 Given a metric space $X$, we consider certain families of functions $f\colon X\to\R$ having the hereditary oscillation property HSOP and the hereditary continuous restriction property HCRP on large sets. When $X$ is Polish, among them there are families of Baire measurable functions, $\overline{\mu}$-measurable functions (for a finite nonatomic Borel measure $\mu$ on $X$) and Marczewski measurable functions. We obtain their characterizations using a class of equivalent point-set games.  In similar aspects, we study cliquish functions, SZ-functions and countably continuous functions.
 \end{abstract}
 \maketitle
\section{Introduction}
In topology several kinds of games between two players have various important applications. They can characterize certain classes of sets or functions in a transparent way. For some survey articles, see \cite{Aurichi,Cao,Gary}.
In our paper we present characterizations of several natural families of real-valued functions defined on a metric space $X$, via point-set games. We continue ideas initiated in \cite{BNS} where point-set games were used to characterize Baire measurable functions on Polish spaces and Lebesgue measurable functions on $\R^k$. The motivation of \cite{BNS} came from the article of Kiss \cite{Kiss} who obtained a new game characterization of Baire~1 functions. Here our approach is more general. We distinguish classes of functions having the hereditary continuous restriction property HCRP and the hereditary small oscillation property HSOP with respect to various kinds of large sets. Two types of point-set games are useful: one of them is $G_1(\Sigma,f)$ that was considered in \cite{BNS} in particular cases, and  the other game $G_{<\omega}(\Sigma, f)$ is a slight modification of $G_1(\Sigma, f)$. 
These are examples of the so-called {\it selection games}, which  are well-represented in the literature; see e.g. \cite{Scheepers}. In particular, a general versions of $G_1(\Sigma,f)$ and $G_{<\omega}(\Sigma, f)$  are considered by Clontz in the recent paper \cite{Clontz}.
We show that these two games are equivalent and, in several settings, determined. However, we leave an open question whether they are determined in general.

The paper is organized as follows. Section~2 is devoted to the definitions of games $G_1(\Sigma,f)$ and $G_{<\omega}(\Sigma, f)$ and their relationships. In particular, we show that these games are equivalent. In Section~3 we introduce properties HCSP and HSOP and we prove the basic Theorem \ref{TW} which describes connections between HCRP and HSOP (for $f\colon X\to\R$ with respect to a family $\Sigma$) and the existence of winning strategies of Players~I and II in the games $G_1(\Sigma, f)$ (so also in  $G_{<\omega}(\Sigma, f)$). In Section~4 we show applications of our results to some kinds of measurable functions having the HSOP with respect to the corresponding families $\Sigma$ of large sets. Other interrelations with the topic are discussed in Sections~5 and~6. This part of our paper is concerned with cliquish functions, SZ-functions and countably continuous functions.

\section{A scheme of games}
Let $X$ be a metric space.
Assume that $\Sigma$ is a non-empty family of non-empty subsets of $X$ and $f\colon X\to\R$ is fixed. 
We will define two point-set games $G_{<\omega}(\Sigma,f)$ and $G_1(\Sigma,f)$. Note that $G_1(\Sigma,f)$ was studied in \cite{BNS}.

The game $G_{<\omega}(\Sigma,f)$ is defined as follows. 
At the initial step of $G_{<\omega}(\Sigma,f)$, Player~I plays ${P}\in\Sigma$, 
then Player II plays a set ${P}_0$.
At the $n$th step, $n\ge 1$,
Player I plays a finite sequence $\langle x_i\colon k_{n-1}<i\le k_n\rangle$
(where $k_0=0$ and $k_n$ is a strictly increasing sequence of naturals),
and Player II plays a ${P}_n$: 
$$\begin{array}{lllllllll}
	\textrm{Player I}  & P &      & x_1,\ldots, x_{k_1}&      & x_{k_1+1},\ldots, x_{k_2} &      & \cdots &        \\
	\textrm{Player II} &     & P_0 &     & P_1 &     & P_2 &        & \cdots
\end{array}$$
with the rules that for each integer $n\ge 0$:
\begin{itemize}
	\item  $x_{i}\in P_n$ for $i=k_n+1,\ldots, k_{n+1}$;
	\item $P_n\in\Sigma$ and $P_n\subseteq {P}$.
\end{itemize}
Player II wins the game $G_{<\infty}(\Sigma,f)$ if 
$\la x_n\ra $ is convergent and $\lim_{n\to\infty} f(x_n)=f(\lim_{n\to\infty}x_n)$. Otherwise, Player~I wins.

The game $G_1(\Sigma,f)$ is similar to $G_{<\omega}(\Sigma,f)$ with one difference: for every $n>0$ we assume that $k_n=n$, i.e. we have
$$\begin{array}{lllllllll}
	\textrm{Player I}  & P &      & x_1 &  & x_2    &      & \cdots        \\
	\textrm{Player II} &     & P_0 &     & P_1 &     & P_2 &     & \cdots
\end{array}$$

\begin{remark}
Note that the finite sequence played by Player~I in $G_{<\omega}$ is non-empty.
This distinguishes this game from the game $G_{<\omega}$ defined in~\cite{Clontz}.
\end{remark}

\begin{remark}\label{remark}
	If Player I  has a winning strategy in the game $G_1(\Sigma,f)$ then he has also a winning strategy in the game $G_{<\omega}(\Sigma,f)$.	
	If Player II  has a winning strategy in the game $G_{<\omega}(\Sigma,f)$, then he has also a winning strategy in the game $G_1(\Sigma,f)$.	
\end{remark}

\begin{proposition} \label{propi}
	Assume that $\Sigma$ is a non-empty family of non-empty subsets of a metric space $(X,d)$ and $f\in\R^X$. If Player I  has a winning strategy in the game $G_{<\omega}(\Sigma,f)$, then he has also a winning strategy in the game $G_1(\Sigma,f)$.
\end{proposition}
\begin{proof}
Let $\$_{<\omega}$ be a winning strategy of Player I in the game $G_{<\omega}(\Sigma,f)$.
We will define a strategy $\$_1$ for Player I in the game $G_1(\Sigma,f)$
and we will show that this strategy is winning.

Fix $n\in\mathbb{N}$ and suppose that $\$_{<\omega}(P_1,P_2,\ldots,P_n)=\la a_1,a_2,\ldots, a_m\ra$, and we have defined $\$_1(P_1,P_2,\ldots,P_k)$ for all $k<n$.
We put $b:=\$_1(P_1,P_2,\ldots,P_{n-1})$ and let
$j_n\le m$ be such that $$\bar{d}(b,a_{j_n})=\max\{ \bar{d}(b,a_j)\colon 1\le j\le m\},$$
where $\bar{d}(x,x'):=d(x,x')+|f(x)-f(x')|$.
Then we define $$\$_1(P_1,\ldots,P_n):= a_{j_n}.$$

To show that $\$_1$ is a winning strategy, suppose to the contrary that there exists
a sequence $\langle P_n\rangle_{n\in\mathbb{N}}$ of sets from $\Sigma$
such that, if $x_n:=\$_1(P_1,P_2,\ldots,P_n)$ for $n\in\N$, then
\begin{equation}\label{prop:win-I:1}
    \lim_{n\to\infty}x_n = x \textup{\ and\ }
    \lim_{n\to\infty}f(x_n) = f(x).
\end{equation}
Since $\$_{<\omega}$ is a winning strategy for Player I in $G_{<\omega}(\Sigma,f)$,
we can fix $\varepsilon>0$ such that, if 
$$\la x'_{k_{n-1}+1},\ldots,x'_{k_n}\ra:=\$_{<\omega}(P_1,P_2,\ldots,P_n),$$
then for each $N\in\mathbb{N}$ there exist $n>N$ and $j\in\{ k_{n-1}+1,\ldots,k_n\}$ such that 
\begin{equation}\label{prop:win-I:2}
\bar{d}(x'_{j}, x)>\varepsilon.
\end{equation}
By (\ref{prop:win-I:1}) we can fix $N_0\in\mathbb{N}$ such that
\begin{equation}\label{prop:win-I:3}
\forall_{n\ge N_0}\;  
\bar{d}(x_n,x)<\frac{\varepsilon}{3}.
\end{equation}
Then, by (\ref{prop:win-I:2}) we can find $m>N_0$ and $j\in\{ k_{m-1}+1,\ldots,k_m\}$ such that
\begin{equation}\label{prop:win-I:4}
\bar{d}(x'_{j},x)>\varepsilon .
\end{equation}
Without loss of generality we can assume that 
\begin{equation}\label{prop:win-I:5}
x_m=x'_{k_{m-1}+1}
\end{equation}
since other cases are analogous.  
Then,  by the formula defining $\$_1$ we obtain
\begin{equation}\label{prop:win-I:6}
\bar{d}(x'_{j},x_{m-1})\le \bar{d}(x_m,x_{m-1}) \textup{ for all\ }j\in\{k_{m-1}+1,\ldots,k_m\},
\end{equation}
and consequently,
\begin{align*}
\frac{\varepsilon}{3} & > \underbrace{\bar{d}(x_m,x)}_{\textup{by (\ref{prop:win-I:3})}}=\underbrace{\bar{d}(x'_{k_{m-1}+1},x)}_{\textup{by (\ref{prop:win-I:5})}}
  \geq \underbrace{\bar{d}(x'_{j},x_{m-1})}_{\textup{by (\ref{prop:win-I:6})}} - \bar{d}(x, x_{m-1})  \geq \\
 & \geq \bar{d}( x'_{j}, x_{m-1}) - \underbrace{\varepsilon/3}_{\textup{by (\ref{prop:win-I:3})}}
  \geq \bar{d}({x'_{j}},x) - \bar{d}( x_{m-1},x) - \varepsilon/3
 > \underbrace{\varepsilon}_{\textup{by (\ref{prop:win-I:4})}} - \underbrace{\varepsilon/3}_{\textup{by (\ref{prop:win-I:3})}} - \varepsilon/3 = \frac{\varepsilon}{3}, 
\end{align*}
which is a desired contradiction.
\end{proof}

 \begin{proposition} \label{propii}
	Assume that $\Sigma$ is a non-empty family of non-empty subsets of a metric space $(X,d)$ and $f\in\R^X$. If Player II  has a winning strategy in the game $G_1(\Sigma,f)$, then he has also a winning strategy in the game $G_{<\omega}(\Sigma,f)$.
\end{proposition}
\begin{proof}
	Let $\$_1$ be a winning strategy of Player II in the game $G_1(\Sigma,f)$. 
	We will define the winning strategy $\$_{<\omega}$ for Player II in the game $G_{<\omega}(\Sigma,f)$.
Let  $\$_{<\omega}(P):=\$_1(P)$. Assume that $n>0$, $P\in\Sigma$, $\{x_1,\ldots,x_{k_n}\}\subseteq X$ and  
$\$_1(P,x_1,\ldots,x_{k_{m-1}}, x_{k_{m-1}+1},\ldots, x_{k_m})$ is defined for all $m<n$. Let $k_0:=1$ and $\bar{x}_0:=x_1$. Then
for every $1\le i\le n$ pick an integer $t$ with $k_{i-1}<t\le k_i$ such that
$$\bar{d}(\bar{x}_{i-1},x_{k_{i-1}+t})=\max\{ \bar{d}(\bar{x}_{i-1},x_{j})\colon k_{i-1}<j\le k_i\},$$
and define $\bar{x}_i:= x_{k_{i-1}+t}$. Then let
$$\$_{<\omega}(P,x_1,\ldots,x_{k_n}):=\$_1(P,\bar{x}_1,\ldots,\bar{x}_n).$$ 

We will verify that $\$_{<\omega}$ is a winning strategy for Player II in the game $G_{<\omega}(\Sigma,f)$.
Suppose to the contrary that there exists a sequence $\la P,x_1,x_2,\ldots,x_{k_n-1},x_{k_n}\ldots\ra$ such that Player I wins the game $G_{<\omega}(\Sigma,f)$ in which he plays $\la x_{k_{n-1}+1},\ldots, x_{k_n}	\ra$ in the $n$th move, and Player II responds with $P_n:=\$_{<\omega}(P,x_1,\ldots,x_{k_n})$. 
Since $\$_1$ is a winning strategy for Player II in the game $G_1(\Sigma,f)$, there exists 
$x\in X$ such that $\lim_n\bar{x}_n=x$ and $\lim_nf(\bar{x}_n)=f(x)$.
But Player I wins the game $G_{<\omega}(\Sigma,f)$ with the moves $\la x_{k_{n-1}+1},\ldots, x_{k_n}\ra$ for $n\in\mathbb{N}$, so either
the sequence $\la x_n\ra$ does not converge to $x$, or $\la f(x_n)\ra$ does not converge to $f(x)$. Therefore, we can fix $\varepsilon>0$ such that for each $N\in\mathbb{N}$ there exists $n>N$ such that 
$$\bar{d}(x_n,x)>\varepsilon.$$
Since $\lim_n\bar{x}_n=x$ and $\lim_nf(\bar{x}_n)=f(x)$, there is $N_0\in\mathbb{N}$ such that for all $n\ge N_0$ we have  
$$\bar{d}(\bar{x}_n,x)<\frac{\varepsilon}{3}.$$ 
Fix $n,m>N_0$ with $\bar{d}(x_n,x)>\varepsilon$ and  $k_{m-1}<n\le k_m$. 
Without loss of generality, we may assume that $\bar{x}_m=x_{k_{m-1}+1}$. Then 
$n=k_{m-1}+j$ for some $1<j\le k_m-k_{m-1}$, and 
we have 
\begin{align*}
\varepsilon< & \bar{d}(x_n,x)= \bar{d}(x_{k_{m-1}+j},x)\le\bar{d}(x_{k_{m-1}+j},\bar{x}_{m-1})+\bar{d}(\bar{x}_{m-1},x) \\
& \le\bar{d}(\bar{x}_m,\bar{x}_{m-1})+\frac{\varepsilon}{3} 
\le \bar{d}(\bar{x}_m,x)+\bar{d}(x,\bar{x}_{m-1})+\frac{\varepsilon}{3} <\varepsilon,
\end{align*}
a contradiction.
\end{proof}

We say that two games~1 and~2 are {\em equivalent} whenever each of players has a winning strategy in 
the game~1 if and only if he has a winning strategy in the game~2.
Thus Propositions \ref{propi} and \ref{propii} yield the following corollary.

\begin{corollary}
Assume that $\Sigma$ is a non-empty family of non-empty subsets of a metric space $(X,d)$ and $f\in\R^X$. Then  the games $G_1(\Sigma,f)$ and $G_{<\omega}(\Sigma,f)$ are equivalent.    
\end{corollary}

We can also consider the following modifications of the games $G_1(\Sigma,f)$ and $G_{<\omega}(\sigma,f)$. 
Let $\lambda$ be an increasing sequence of natural numbers. 
The definition of the game $G_\lambda(\Sigma,f)$ is the same as the definition of $G_{<\omega}(\Sigma,f)$ with the condition: $k_n=\lambda(n)$ for every $n\in\mathbb{N}$. Moreover, if $\lambda(n)=mn$, the game $G_{\lambda}(\Sigma,f)$ is denoted by $G_m(\Sigma,f)$, i.e. 
at the $n$th step of the game $G_m(\Sigma,f)$, 
where $n\ge 1$,
Player~I plays a finite sequence $\langle x_i\colon m(n-1)<i\le mn\rangle$ and Player~II plays a set ${P}_n$  as follows: 
$$\begin{array}{lllllllll}
	\textrm{Player I}  & P &      & x_1,\ldots, x_{m}&      & x_{m+1},\ldots, x_{2m} &      & \cdots &        \\
	\textrm{Player II} &     & P_0 &     & P_1 &     & P_2 &        & \cdots
\end{array}$$
It is easy to observe that:
\begin{itemize}
    \item 
    If Player I  has a winning strategy in the game $G_1(\Sigma,f)$ then he has also a winning strategy in the game $G_{\lambda}(\Sigma,f)$, for any increasing sequence $\lambda$.	
    \item
	If Player II  has a winning strategy in the game $G_{<\omega}(\Sigma,f)$, then he has also a winning strategy in the game $G_\lambda(\Sigma,f)$ for every $\lambda$.
\end{itemize}
Therefore, all games $G_\lambda(\Sigma,f)$ are equivalent to the game $G_1(\Sigma,f)$ (so also to the game $G_{<\omega}(\Sigma,f)$).

In the next sections we give several examples where the games $G_1(\Sigma, f)$ (so, also $G_{<\omega}(\Sigma,f))$  are determined for functions $f$ from special classes of functions with the respectively chosen $\Sigma$.
The following question seems to be interesting in this context.
\begin{problem}
	Do there exist a family $\Sigma$ of non-empty subsets of $\R$ and a function $f\colon\R\to\R$ for which the game $G_1(\Sigma,f)$ is not determined?
\end{problem}

\section{Functions with the hereditary small oscillation property on large sets}

Let $g\colon A\to\R$ with $A\subseteq X$.
Recall the notion of the oscillation $\osc(g,x)$ of $g$ at a point $x\in A$. Namely, let
$$\osc(g,x):=\lim_{\ve\to 0^+}\on{diam}(g(A\cap B(x,\ve)) $$
where $B(x,\ve)$ denotes the open ball centered at $x$ with radius $\ve$.
We know that $g$ is continuous at $x\in A$ if and only if $\osc(g,x)=0$.

The collection $\Sigma$ will play a role of large subsets of $X$. We say that a family $\Sigma$
is {\em  dense} if, for each $P\in\Sigma$ and every ball $B(x,r)$ with $x\in P$ there exists $Q\in\Sigma$
contained in $P\cap B(x,r)$. 

Given a family of non-empty subsets of $X$, we say that a function $f\colon X\to\R$ has:
\begin{itemize}
\item the {\em hereditary continuous restriction property} ($\HCRP$) with respect to $\Sigma$ whenever
for each $P\in\Sigma$ there exists a subset $Q\subseteq P$ in $\Sigma$ such that $f\res Q$ is continuous;
\item the {\em hereditary small oscillation property} ($\HSOP$) with respect to $\Sigma$ whenever
for every $\alpha >0$ and each $P\in\Sigma$ there exists a subset $Q\subseteq P$ in $\Sigma$ such that 
$\osc(f\res Q,x)<\alpha$ for all $x\in Q$.
\end{itemize}
We say that a family $\mc F\subseteq \R^X$ has the $\HCRP$ (respectively, the $\HSOP$) with respect to $\Sigma$ whenever every function $f$ in $\mc F$ has a property with the same name. We use this terminology to simplify the notation in our results. Note that the HCRP was used in \cite{BET} in an implicit way without this special notation. 
It is a stronger version of the {\it continuous restrictions property} (CRP) introduced by Recław in \cite{Reclaw}.

Clearly, the hereditary continuous restriction property for $f$ implies its hereditary small oscillation property.
In the next section, we will discuss some examples for which the reverse implication holds.

Now, our aim is to prove the following general result. 
We will consider the following statements:
\begin{itemize}
	\item[(i)] 
	$f$ has the $\HCRP$ with respect to $\Sigma$;
	\item[(ii)]
	$f$ has the $\HSOP$ with respect to $\Sigma$;
	\item[(iii)]
	Player II has a winning strategy in the game $G_1(\Sigma,f)$;
	\item[(iv)]
	Player I has a winning strategy in the game $G_1(\Sigma,f)$.
\end{itemize}

\begin{theorem} \label{TW}
Assume that $(X,d)$ is a complete metric space and $\Sigma$ is a dense family whose members are non-empty closed subsets of $X$.
	Then  {\em (ii)} $\Leftrightarrow$ {\em (iii)} 
 and {\em $\neg$(ii)} $\Leftrightarrow$ {\em (iv)}. 
	This means that the game $G_1(\Sigma,f)$ is determined in this case.
	\end{theorem}
	
	\begin{proof}
Fix $f\colon X\to\R$. 

``(ii) $\Rightarrow$ (iii)''. 
Assume that $f$ has the $\HSOP$ with respect to $\Sigma$.
Let Player I choose $P\in\Sigma$. Let Player II pick a subset $P_0\in\Sigma$ of $P$ such that $\osc(f\res P_0,x)<1$ for $x\in P_0$.  Assume that Player I has made his $k$th move. So, points $x_1,x_2,\dots,x_k$ have been chosen. Since the family $\Sigma$ is dense, there is $Q\in\Sigma$ such that
$$Q\subseteq P_{k-1}\cap B\left(x_{k},\frac{1}{2^k}\right).$$
By the assumption on $f$, there is $P_{k}\subseteq Q$ in $\Sigma$ with $|f(x)-f(x')|<\frac{1}{k}$ for all $x,x'\in P_{k}$. Let Player II play  $P_{k}$ at his $k$th move.
Then Player I will choose $x_{k+1}\in P_{k}$ and so on.

Observe that $\la x_i\ra$ is a Cauchy sequence. Indeed, fix positive integers $N$ and $k>N$. Then 
$$d(x_{N},x_{k}) \le d(x_{N},x_{N+1})+ d(x_{N+1},x_{N+2})+d(x_{N+2},x_{N+3}) +$$
$$\dots+d( x_{k-1},x_{k}).$$
Since for $j=N, N+1,\dots, k$ we have $x_j,x_{j+1}\in B(x_j,\frac{1}{2^j})$, so
$$d(x_{j},x_{j+1})<\frac{2}{2^j}$$
and
$$d(x_{N},x_{k}) \le\sum_{j=N}^\infty \frac{1}{2^j}<\frac{1}{2^{N-1}}$$
The remaining details are left to the reader.

Since $X$ is complete, there exists $x\in X$ such that $\lim_{i} x_i=x$. Since the sets $P_{k}$ are closed,  $x\in\bigcap_{k}P_{k}$, so $|f(x)-f(x_i)|<\frac{1}{k}$ for $i>k$ and consequently, $\lim_{i} f(x_i)=f(x)$.


``$\neg$(ii) $\Rightarrow$ (iv)''.
Suppose that $f$ has not the $\HSOP$ with respect to $\Sigma$. So, there exist a set $P\in\Sigma$ and $\alpha>0$ such that, for each $Q\subseteq P$ in $\Sigma$, there exists $x\in Q$ with $\on{osc}(f\res Q,x)\ge\alpha$. Let Player~I play as follows.
Firstly, he chooses $P$ as above. Next Player~II picks $P_0\in\Sigma$, $P_0\subseteq P$, and Player~I chooses any $x_1\in P_0$. Assume that for some integer $k> 0$ a point $x_{k-1}$ and a set $P_{k}$ have been chosen. 
Then Player~I picks any $x\in P_k$ and he considers two cases. If $|f(x)-f(x_{k-1})|>\frac{\alpha}{4}$ then he picks $x_k=x$. Otherwise, $|f(x)-f(x_{k-1})|\le\frac{\alpha}{4}$ and since $\osc(f\res P_k,x)\ge\alpha$, there exists $x'\in P_k$ with $|f(x)-f(x')|\ge\frac{\alpha}{2}$. Then $|f(x')-f(x_{k-1})|\ge \frac{\alpha}{4}$ and Player~II chooses $x_k=x'$. 
It guarantees that $|f(x_{k-1})-f(x_k)|\ge\frac{\alpha}{4}$ for any $k\in\N$. Hence the sequence $\la f(x_n)\ra$ is not convergent and Player~I wins the game $G_1(\Sigma,f)$. Thus Player I has the winning strategy in this game.	
%
%
%

Finally, observe that in both cases, (ii) or $\neg$(ii), one of the players
has a winninig strategy in the game $G_1(\Sigma,f)$, so the game is determined.
\end{proof}

\begin{remark}\label{remark2}
Notice that the assumption that any set in $\Sigma$ is closed and $X$ is complete is only used in the proof of the implication (ii) $\Rightarrow$ (iii).
\end{remark}

Since the implication (i) $\Rightarrow$ (ii) holds for any $f\colon X\to\R$, we have the following fact.

\begin{corollary}
Assume that $(X,d)$ is a complete metric space, $\Sigma$ is a dense family whose members are non-empty closed subsets of $X$. If 
the implication (ii) $\Rightarrow$ (i) holds for a function $f\colon X\to\R$, then the statements 
(i),(ii),(iii),$\neg$(iv) are equivalent.
\end{corollary}

\section{Applications}
Here we discuss some examples of families of functions with the hereditary continuous restriction property on large sets.
Three of them come from the article \cite{BET} where the authors considered characterizations by this property for certain functions that are measurable with respect to some known $\sigma$-algebras of sets.

Now, let $X$ be a Polish space without isolated points. 
By $\Perf$ we denote the family of all nonempty perfect subsets of $X$.
Consider three families of measurable functions.

\subsection*{Baire measurable functions} 
These are functions $f\colon X\to\R$ which are measurable with respect to $\sigma$-algebra of sets with the Baire property. It is well known that if $X$ is complete then Baire measurable functions are exactly those functions $f$ whose restrictions $f\res G$ to a residual (comeager) set $G$ 
are continuous (see e.g. \cite[Thm 8.38]{Ke}). It was proved in \cite[Thm 9]{BET} that $f$ is Baire measurable if and only if it has the $\HCRP$ with respect to $G_\delta$ sets $G\subseteq X$ having the property: there is a nonempty open set $U$ such that $G$ is residual in $U$. The family of these $G_\delta$ sets will be denoted by $\on{G_{Res}}$.
	
\subsection*{Functions measurable with respect to a measure}
Let $\mc M$ be the collection of sets which are measurable with respect to the completion $\overline{\mu}$ of a finite nonatomic Borel measure $\mu$ on $X$. and let $\Perf^+$ stand for the family of all perfect subsets of $X$ with positive measure. It was proved in \cite[Thm 8]{BET} that $f\colon X\to\R$ is $\mc M$-measurable if and only if it has the $\HCRP$ with respect to $\Perf^+$.

\subsection*{Marczewski measurable functions} These are functions $f\colon X\to\R$ which are measurable with respect to $\sigma$-algebra of (s)-sets $E\subseteq X$ with the property: for each $P\in\Perf$ there exists $Q\in\Perf$ such that either $Q\subseteq P\cap E$ or $Q\subseteq P\setminus E$, cf. \cite{Sz} and \cite{BET}. In \cite{Sz} the following nice characterization was established (see also \cite{BET}): a function $f$ is 
$(s)$-measurable if and only if it has the $\HCRP$ with respect to $\Perf$.

\subsection*{HSOP versus HCRP}
We are going to prove that for the above kinds of measurable functions, the $\HSOP$ is equivalent to the respective $\HCRP$.

\begin{proposition} \label{propa}
Let $X$ be a Polish space without isolated points and $f\colon X\to\R$. Then:
\begin{itemize}
\item[(a)] if $f$ has the $\HSOP$ with respect to $\Perf$, then it is Marczewski measurable;
\item[(b)] given a finite nonatomic Borel measure $\mu$ on $X$, if $f$ has the $\HSOP$ with respect to $\Perf^+$, then it is measurable with respect to $\overline{\mu}$;
\item[(c)] if $f$ has the $\HSOP$ with respect to $\on{G_{Res}}$, then it is Baire measurable.
\end{itemize}
Consequently, in all these cases, the $\HSOP$ is equivalent to the respective $\HCRP$ for $f$.
\end{proposition}
\begin{proof}
In fact, we will use the above-mentioned characterizations of the three kinds of measurability. So, we
will show that the $\HSOP$ implies the respective $\HCRP$, as it is stated in the final assertion. 

Ad (a). Let $P\in\on{Perf}$. Using the assumption, pick a perfect subset $Q$ of $P$ such that $\on{osc}(f\res Q,x)<1$ for each $x\in Q$. Then find two disjoint perfect subsets
$Q_{\la 0\ra}$, $Q_{\la 1\ra}$ of $Q$, of diameter less than $1$. 
Assume that for $n\in\N$, perfect pairwise disjoint sets $Q_s$, $s\in\{0,1\}^n$, have been chosen where each of them has diameter less than $1/n$, and $\on{osc}(f\res Q_s, x)<1/n$ for every $x\in Q_s$. For any fixed $Q_s$ find two disjoint perfect subsets $Q_{s\frown 0}$ and $Q_{s\frown 1}$ of $Q_s$, of diameter less than $1/(n+1)$, with $\on{osc}(f\res Q_{s\frown i}, x)<1/(n+1)$ for all $x\in Q_{s\frown i}$ and $i\in\{0,1\}$. Finally, define
$$Q^\star:=\bigcap_{n\in\N}\bigcup_{s\in\{0,1\}^n} Q_s.$$
Note that $Q^\star$ is a perfect subset of $P$ and $f\res Q^\star$ is continuous since its oscillation at every point of $Q^\star$ is zero. 

Ad (b). We should be more careful in comparison with the previous case. Fix a perfect set $P$ of positive measure. In the first step, using the assumption, we consider a maximal disjoint family $\mc J_1$ of perfect sets $Q\subseteq P$ of positive measure such that $\osc(f\res Q,x)<1$ for all $Q\in\mc J_1$ and $x\in Q$. Note that $\mc J_1$ is countable and $\mu(\bigcup {\mc J_1})=\mu(P)$.
Pick a finite family ${\mc J}_1^*\subseteq\mc J_1$ such that for $S_1:=\bigcup{\mc J_1^*}$ we have
$\mu(P\setminus S_1)<\mu(P)/2^2$. Then $\osc(f\res S_1,x)<1$ for each $x\in S_1$. In the next step, in a similar way, we select finite disjoint family of perfect subsets of $S_1$ whose union, called $S_2$, satisfies the conditions $\mu(S_1\setminus S_2)<\mu(P)/2^3$ and $\osc(f\res S_2,x)<1/2$ 
for each $x\in S_2$. In such a way we obtain a decreasing sequence of perfect sets $S_n$ such that for every $n\in\N$ we have
$$\mu(S_n\setminus S_{n+1})<\mu(P)/2^{n+2}\;\;\mbox{  and }\;\;\osc(f\res S_n,x)<\frac{1}{n}\;\;\mbox{ for each}\;\; x\in S_n.$$
Then note that the set $S:=\bigcap_{n\in\N}S_n$ is closed and $\mu(S)>\mu(P)/2$. Also, $\osc(f\res S, x)=0$ for each $x\in S$. Finally, we can select a perfect set $Q\subseteq S$ of positive measure. Then $f\res Q$ is continuous. 

Ad (c). Fix a $G_\delta$ set $G$ that is contained in a nonempty open set $U$ where $M:=U\setminus G$ is meager (that is, of the first category) of type $F_\sigma$. Using the assumption, we can find, for each $k\in\N$,  a maximal disjoint family $\mc J_k$ of open subsets of $U$ such that for every $V\in \mc J_\red{k}$ there is a meager set $M_V\subseteq V$ of type $F_\sigma$ with $\osc(f\res (V\setminus M_V), x)<1/k$ 
for each $x\in V\setminus M_V$. Observe that for each $k\in\N$,
\begin{itemize}
\item the family $\mc J_k$ is countable;
\item the set $\bigcup{\mc J_k}$ is open and dense in $U$, so $N_k:= U\setminus\bigcup\mc J_k$ is nowhere dense of type $F_\sigma$;
\item the set $W_k:=N_k\cup\bigcup_{V\in\mc J_k}M_V$ is meager of type $F_\sigma$, and 
$\osc(f\res (G\setminus W_k), x)<1/k$ for each $x\in G\setminus W_k$.  
\end{itemize}
Next, we define an increasing sequence of meager sets $T_n$ (for $n\in\N$) of type $F_\sigma$, contained in $U$, such that
$\osc(f\res (U\setminus T_n), x)<1/n$ for each $n\in U\setminus T_k$. Namely, let $T_1:=M\cup W_1$, and
$T_{n+1}:=T_n\cup\bigcup_{i\leq n} W_i$. Finally, observe that $T:=\bigcup_{n\in\N}T_n$ is a meager subset of $U$,
of type $F_\sigma$ and $H:=U\setminus T$ is a set, residual in $U$, of type $G_\delta$, contained in $G$, with $\osc(f\res H,x)=0$ for each $x\in H$, as desired.
\end{proof}

\begin{remark}
Statements (b) and (c) in the above proposition can also be deduced from~\cite[Lemma 10]{BNS}. However, for the reader's convenience, we have decided to include here the direct proofs of these facts.
\end{remark}

Observe that, if $X$ is a Polish space without isolated points, the families $\on{G_{Res}}$, $\Perf^+$ and 
$\Perf$ are dense. Additionally, $\Perf^+$ and $\Perf$ consist of closed sets.
So, by Proposition \ref{propa} and Theorem \ref{TW}, we can infer the following corollaries.

\begin{corollary} \label{c1}
Let $X$ be a Polish space without isolated points. Let $f\colon X\to\R$.
Given a finite nonatomic Borel measure $\mu$ on $X$, let $\Sigma:=\Perf^+$. Then the game $G_1(\Sigma,f)$ is determined and 
\begin{itemize}
    \item 
    if $f$ is measurable with respect to the completion of $\mu$ then Player~II has a winning strategy in the game $G_1(\Sigma,f)$,
    \item 
    if $f$ s not measurable with respect to the completion of $\mu$ then Player~I has a winning strategy in the game $G_1(\Sigma,f)$.
    \end{itemize}

\end{corollary}

Note that the 
the same result was established in \cite{BNS} in a bit different manner.

\begin{corollary} \label{c2}
Let $X$ be a Polish space without isolated points. Let $f\colon X\to\R$
and $\Sigma:=\Perf$.
Then  the game $G_1(\Sigma,f)$ is determined and 
\begin{itemize}
    \item 
    if $f$ is Marczewski measurable  then Player II has a winning strategy in the game $G_1(\Sigma,f)$,
    \item 
    if $f$ is not Marczewski measurable  then Player I has a winning strategy in the game $G_1(\Sigma,f)$.
    \end{itemize}

\end{corollary}

We have a small problem with Baire measurable functions since the respective family $\on{G_{Res}}$
consists of $G_\delta$ sets, and the closedness of sets was used in the proof of Theorem \ref{TW} (cf. Remark~\ref{remark2}). We will solve this problem by the following modification.

\begin{corollary} \label{c3}
	Let $X$ be a Polish space without isolated points. Let $f\colon X\to\R$
	and $\Sigma:=\on{G_{Res}}$. Then the game $G_1(\Sigma,f)$ is determined and 
 \begin{itemize}
    \item 
    if $f$ is Baire measurable  then Player II has a winning strategy in the game $G_1(\Sigma,f)$,
    \item 
    if $f$ is not Baire measurable  then Player I has a winning strategy in the game $G_1(\Sigma,f)$.
    \end{itemize}
    
\end{corollary}
\begin{proof}
By Remark \ref{remark2}, only the first implication 
has to be proved. Assume that $f\colon X\to\R$ has the Baire property. Then there is a meager $F_\sigma$ set $M\subseteq  X$ such that $f\res (X\setminus M)$  is continuous. Let $M=\bigcup_{n\in\N}F_n$, where all sets $F_n$ are closed and $F_n\subseteq F_{n+1}$ for any $n\in\N$.
Let Player I choose $P\in \on{G_{Res}}$. Assume that $P$ is residual in a non-empty open set $U\subseteq X$. 
Let Player II pick a subset $P_0\in\on{G_{Res}}$ of $P\setminus M$ such that $\cl(P_0)\cap F_1=\emptyset$.  Assume that Player I has made his $k$th move. So, points $x_1,x_2,\dots,x_k$ have been chosen. Then Player II plays a set $P_{k}\in\on{G_{Res}}$ such that $P_{k}\subseteq P_{k-1}\cap B\left(x_{k},\frac{1}{2^k}\right)$, $\diam(P_{k})<\frac{1}{2^k}$, and $\cl(P_{k})\cap F_k=\emptyset$.
Then Player I will choose $x_{k+1}\in P_{k}$ and so on. As in the proof of Theorem \ref{TW}, one can verify that the sequence $\la x_i\ra$ converges to some $x\in X$. Note that $x_i\in\cl(P_{k})$ for every $i>k$, hence $x\in\bigcap_{i\in\N}\cl(P_{i})$ and $\bigcap_{i\in\N}\cl(P_{i})\cap M=\emptyset$,  so $x\in X\setminus M$. Since all points $x_i$ belong to $P_0\subseteq X\setminus M$ and  $f\res (X\setminus M)$ is continuous at $x$, so $\lim_{i}f(x_i)=f(x)$. 
\end{proof}

Note that in \cite{BNS} the Baire measurability of a function was characterized in the language of the winning strategy of Player II in the game $G_1(\Sigma,f)$ where $\Sigma$ consists of Baire nonmeager sets.

\section{Cliquish functions}
Let $X$ be a topological space. Recall that a function $f\colon X\to\R$ is:
\begin{itemize}
	\item {\it cliquish} if for each non-empty open set $W$ and $\varepsilon>0$ there is a non-empty open set $U\subseteq W$ with $\diam(f(U))<\varepsilon$; see \cite{Thiel};
	\item {\it pointwise discontinuous} if the set $C(f)$ of continuity points of $f$ is dense in $X$; see \cite{KKur}. 
\end{itemize}
It is well-known and easy to show that each pointwise discontinuous function is also cliquish. Moreover, if $X$ is a Baire space (i.e., $X$ is non-meager in itself) then those notions are equivalent.
The following fact results directly from the definition.

\begin{lemma}\label{lemma}
Assume that $X$ is a regular ($T_3$) space. Then $f\colon X\to\R$ is cliquish if and only if for every $\ve >0$ and for each non-empty open set $W\subseteq X$ there is a non-empty open set $U\subseteq W$ with $\diam(f(\cl(U))<\varepsilon$.
\end{lemma}

Let $\mathrm{CLO}(X)$ denote the family of all closed subsets of $X$ possessing non-empty interior.
\begin{corollary}
	Assume that $X$ is a complete metric space and $\Sigma=\mathrm{CLO}(X)$. Then  the game $G_1(\Sigma,f)$ is determined and 
 \begin{itemize}
    \item 
    if $f$ is cliquish  then Player II has a winning strategy in the game $G_1(\Sigma,f)$,
    \item 
    if $f$ is not cliquish  then Player I has a winning strategy in the game $G_1(\Sigma,f)$.
    \end{itemize}
	\end{corollary}
	\begin{proof}
	By Lemma \ref{lemma},  $f\colon X\to\R$ is cliquish if and only if it has the HSOP with respect to $\mathrm{CLO}(X)$. Now the assertion holds by Theorem  \ref{TW}.
	\end{proof}

	Observe that a function $f\colon X\to\R$ has the $\HCRP$  with respect to $\mathrm{CLO}(X)$ if and only if the set $\int(C(f))$ is dense in $X$. Therefore, unlike previous examples, in this case the $\HSOP$ and the $\HCRP$ are not equivalent.
\begin{example}
	Let $X=\R$. List all rationals in the sequence  $\{ q_n\colon n\in\N\}$. Let $f\colon\R\to\R$ be defined by $f(x)=\frac{1}{n}$ for $x=q_n$, $n\in\N$, and $f(x)=0$ for $x\in\R\setminus\Q$. Then the set $C(f)=\R\setminus\Q$ is dense in $\R$, hence $f$ is cliquish but $\int(C(f))=\emptyset$, so $f$ has not the $\HCRP$ with respect to $\mathrm{CLO}(\R)$. 
\end{example}

\section{Countably continuous functions and SZ-functions}
Assume that $X$ is a separable metric space.
A function $f\colon X\to\R$ is called:
\begin{itemize}
	\item {\it $\kappa$-continuous} if there exists a  partition $\{ X_\alpha\colon \alpha<\kappa\}$ such that $f\res X_\alpha$ is continuous for any $\alpha<\kappa$;
	\item
	{\it $<\kappa$-continuous} if it is $\lambda$-continuous for any $\lambda<\kappa$;
	\item
	{\it Sierpi\'{n}ski-Zygmund}  (in short, SZ) if $f\res A$ is continuous for no $A\subseteq X$ with $\card(A)=\Co$ (the cardinality of $\R$);
	\item
	{\it somewhere Sierpi\'{n}ski-Zygmund} if there is $X_0\subseteq X$ such that $\card(X_0)=\Co$ and $f\res X_0$ is SZ on $X_0$.
	\end{itemize}
 In particular, $\omega$-continuous functions are called {\it countably continuous}. 
Notice that the notions of $\omega$-continuity and $<\Co$-continuity coincide under CH.
The following fact characterizes the $<\Co$-continuous functions
under the assumption that $\Co$ is a regular cardinal. This result generalizes a characterization of countable continuity formulated (for $X=\R$, under CH, and without a
proof) by Darji in \cite[Theorem 10]{Darji}. 
\begin{fact}\label{cc}
	\mbox{\cite[Proposition 2.3]{CN}}
	Assume that $\Co$ is a regular cardinal and $X$ is separable and metric. Then for any $f\colon X\to\R$ the following conditions are equivalent:
	\begin{enumerate}
		\item $f$ is $<\Co$-continuous;
		\item for every $U\subseteq X$ with $\card(U)=\Co$ there exists a set $V\subseteq U$ with $\card(V)=\Co$ such that $f\res V$ is continuous, i.e. $f$ has the $\HCRP$ with respect to the family $\Sigma=[X]^\Co$;
		\item
		$f$ is nowhere SZ, i.e. it is not somewhere SZ.
	\end{enumerate}
\end{fact}

\begin{proposition}\label{p}
	Assume that $X$ is a metric space with  $|X|\ge\Co$. Then every function $f\in\R^X$ has the $\HSOP$ with respect to the family $\Sigma=[X]^{\Co}$. This means that if $X$ is additionally separable, then properties $\HCRP$ and $\HSOP$ with respect to $\Sigma$ are not equivalent.
\end{proposition}
\begin{proof}
	Fix $f\in\R^X$, $\alpha>0$, and $A\in [X]^{\ge \Co}$. We may assume that $|A|=\Co$. Let $n\in\N$ be such that $\frac{1}{n}<\frac{\alpha}{2}$ and, for each $k\in\Z$ let $J_{k,n}=[\frac{k}{n},\frac{k+1}{n})$ and $A_{k,n}=f^{-1}[J_{k,n}]$. Then $A=\bigcup_{k\in\Z}A_{k,n}$, so $|A_{k,n}|=\Co$ for some $k$. Fix any $B\subseteq A_{k,n}$. Then $|f(x)-f(x')|<\frac{1}{k}<\frac{\alpha}{2}$ for any $x,x'\in B$, hence $\osc(f\res B,x)<\alpha$ for each $x\in B$.
	
	The second part of the assertion follows from the fact that if $X$ is separable then there exists a SZ-function $f\in \R^X$, cf. \cite{CN}, and any such an $f$ has the $\HSOP$ but it has not the $\HCRP$ with respect to $\Sigma$.
	\end{proof}

\begin{corollary}
	Assume that $\Co$ is a regular cardinal, $X$ is an uncountable Polish space and $\Sigma=[X]^\Co$.
	Then for any $f\in \R^X$, Player II has a winning strategy in the game $G_1(\Sigma,f)$. 
\end{corollary}

\begin{proof}
At the beginning of the play, Player I chooses a set $P\in [X]^\Co$. For every $n\in\N$ let $\mathcal{A}_n$ denote a maximal family of pairwise disjoint  subsets $A\subset P$ such that $|A\cap B(x,\varepsilon)|=\Co$ for every $\varepsilon>0$ and $x\in A$ (i.e., $A$ is $c$-dense-in-itself) and  $|f(x)-f(x')|<\frac{1}{n}$ for $x,x'\in A$. 
Observe that $B_n:=P\setminus\bigcup\mathcal{A}_n$ is of size less than $\Co$. Let $B:=\bigcup_nB_n$. 
Since $\Co$ is a regular cardinal, 
$A:=P\setminus B$ has cardinality $\Co$, so there exists $x_0\in A$ such that $|A\cap B(x_0,\varepsilon)|=\Co$ for every $\varepsilon>0$. For every $k\in\N$ there exists $A_k\in\mathcal{A}_k$ with $x_0\in A_k$. Now, Player~II chooses in his $k$th move the set $P_{k}:=A_k\cap B(x_0,\frac{1}{k})$. Then, regardless of the choice of  points $x_k \in P_{k-1}$  by Player~I, we have $\lim_{n}x_n=x_0$ and  $|f(x)-f(x_0)|<\frac{1}{k}$, hence $\lim_{n}f(x_n)=f(x_0)$. 
To conclude, the above is a winning strategy for Player~II.
\end{proof}


\bibliographystyle{abbrv}
\bibliography{games}

\end{document}